\documentclass[12pt]{amsart}
\usepackage{etex}
\reserveinserts{28}

\usepackage[T1]{fontenc}
\usepackage[cp1250]{inputenc}
\usepackage[margin=1in, marginpar=0.8in]{geometry}
\usepackage{color}
\usepackage[table]{xcolor}
\usepackage{adjustbox}
\usepackage[centertags]{amsmath}
\usepackage{amsfonts}
\usepackage{amsthm}
\usepackage{amsgen}
\usepackage{amssymb}
\usepackage{amssymb}
\usepackage{amscd}
\usepackage{enumerate}
\usepackage{color}
\usepackage{longtable}
\usepackage{multicol}
\usepackage{graphicx}
\usepackage{colordvi}
\usepackage{booktabs}

\usepackage[all,  colour, line, dvips]{xy}
\usepackage{textcomp}
\newlength{\defbaselineskip} 
\usepackage[normalem]{ulem}
\theoremstyle{plain}
\newtheorem{thm}{Theorem}[section]
\newtheorem{cor}[thm]{Corollary}
\newtheorem{con}[thm]{Conjecture}

\newtheorem{lema}[thm]{Lemma}
\newtheorem{prop}[thm]{Proposition}
\newtheorem{obs}[thm]{Proposition}

\newtheorem{fact}[thm]{Fact}

\theoremstyle{remark}
\newtheorem{rem}[thm]{Remark}
\newtheorem{pr}{Algorithm}
\newtheorem{exm}[thm]{Example}
\theoremstyle{definition}  \newtheorem{df}[thm]{Definition}
\theoremstyle{definition} \newtheorem{ex}{Example}[section] %

 \numberwithin{equation}{section}

\def\Z{\mathbb{Z}}
\def\N{\mathbb{N}}

\def\fa{\begin{fact}}
\def\kfa{\end{fact}}

\def\R\mathcal{R}
\def\NZQ{\mathbb}
\def\QQ{{\NZQ Q}}

\def\eb{{\bold e}}

\DeclareMathOperator{\Imi}{Im}
\DeclareMathOperator{\Rel}{Re}

\DeclareMathOperator{\Spec}{Spec}

\def\ob{\begin{obs}}
\def\kob{\end{obs}}
\def\dow{\begin{proof}}
\def\kdow{\end{proof}}

\def\tw{\begin{thm}}
\def\ktw{\end{thm}}
\def\hip{\begin{con}}
\def\khip{\end{con}}
\def\lem{\begin{lema}}
\def\klem{\end{lema}}
\def\ex{\begin{exm}}
\def\prog{\begin{pr}}
\def\kprog{\end{pr}}
\def\wn{\begin{cor}}
\def\kwn{\end{cor}}
\def\uwa{\begin{rem}}
\def\kuwa{\end{rem}}
\def\kex{\end{exm}}
\def\dfi{\begin{df}}
\def\kdfi{\end{df}}
\definecolor{zielony}{rgb}{0.5, 0.9, 0.1}
\definecolor{czerwony}{rgb}{0.9, 0.2, 0.1}
\definecolor{niebieski}{rgb}{0.3, 0.1, 0.9}

\newenvironment{red}{\color{czerwony}}{}
\newcommand{\bred}{\begin{red}}
 \newcommand{\ered}{\end{red}}

\title{Interlacing Ehrhart Polynomials of Reflexive Polytopes}
\author[A.~Higashitani]{Akihiro Higashitani}
 \thanks{The first author is partially supported by JSPS Grant-in-Aid for Young Scientists (B) 26800015}
\address{Department of Mathematics,
Kyoto Sangyo University, Motoyama, Kamigamo, Kita-ku, Kyoto 603-8555, Japan}
\email{ahigashi@cc.kyoto-su.ac.jp}

 \author[M.~Kummer]{Mario Kummer}
 \thanks{}
 \address{Max Planck Institute for Mathematics in the Sciences\\
 Inselstra\ss{}e 22\\
 04103 Leipzig, Germany}
 \email{kummer@mis.mpg.de}

\author[M.~Micha{\l}ek]{Mateusz Micha{\l}ek}
 \thanks{The third author is supported by the Polish National Science Center grant 2013/08/A/ST1/00804}
 \address{Institute of Mathematics of
 Polish Academy of Sciences\\
 ul. \'Sniadeckich 8\\
 00-956 Warszawa, Poland}
 \email{wajcha2@poczta.onet.pl}
 
\newcommand{\RR}{\mathbb{R}}
\newcommand{\CC}{\mathbb{C}}

\newcommand{\ZZ}{\mathbb{Z}}

\definecolor{Zielony}{RGB}{78,104,18} 
\definecolor{JasnoZielony}{RGB}{195,204,110} 

\usepackage[pdfborder={0 0 0}]{hyperref}

\begin{document}
\begin{abstract}
It was observed by Bump et al.~that Ehrhart polynomials in a special family exhibit properties similar to the Riemann $\zeta$ function. The construction was generalized by Matsui et al.~to a larger family of reflexive polytopes coming from graphs. We prove several conjectures confirming when such polynomials have zeros on a certain line in the complex plane. Our main new method is to prove a stronger property called \emph{interlacing}.
\end{abstract}

\maketitle

\section{Introduction}
The aim of the article is to investigate relations among three classical mathematical objects: graphs, polytopes and polynomials.
There are known constructions that associate to a graph a lattice polytope - \emph{the symmetric edge polytope} \cite{MHNOH, ohsugi2014centrally, ohsugi2012smooth}. Furthermore, to each lattice polytope $P$ one associates the \emph{Ehrhart polynomial} $H_P$ \cite{ehrhart1962geometrie, beck2007computing, bruns2009polytopes, stanley1997enumerative} that computes the number of lattice points in dilations of $P$. This is also the Hilbert polynomial of the normal toric variety $\Spec \CC[C]$, where $C$ is the cone over $P$ \cite{Stks, cox2011toric}. Both constructions are briefly recalled in Section \ref{sec:PPG}. 

Furthermore, the roots of Ehrhart polynomials are also an object of intensive studies \cite{braun2008norm, beck2005coefficients, bey2007notes, higashitani2012counterexamples}.
In case of many graphs these roots have a remarkable property: They lie on a line $R=\{z:\Rel(z)=-\frac{1}{2}\}$. Still, proving that this property holds for a family of graphs is often very hard. One of the first positive results was the case of the complete $(1,n)$-bipartite graphs (trees) proved independently by Kirschenhofer, Peth{\"o} and Tichy \cite[Thm.~3.4]{kirschenhofer1999algebra} and by Bump, Choi, Kurlberg and Vaaler \cite[Thm.~4, Thm.~6]{BCKV}. In the former, the authors studied this family of polynomials in relation to finiteness results on the number of solutions of diophantine equations and in relation to Meixner polynomials.
In the latter this family of polynomials was studied in the context of the local Riemann hypothesis. Indeed, while the polynomials $f$ that we associate to trees are symmetric with respect to $R$, the polynomials $f(-x)$ appear as Mellin's transforms of Laguerre functions and have properties similar to the Riemann $\zeta$ function - cf.~\cite[Ch.~3]{BCKV}. Related results were obtained by Rodriguez-Villegas, who further speculated about relations to Riemann hypothesis \cite[Ch.~4]{rod02}. In fact, the study of functions with similar properties to $\zeta$ (but simpler) and their zero loci goes back even to P{\'o}lya \cite{polya1926bemerkung} and some of his methods apply to particular cases of polynomials we study \cite[p.~4,5]{BCKV}.

While our methods work easily for trees, they extend to many other classes. 
Our new approach is based on the observation that the polynomials we obtain come in families with roots having stronger properties than just belonging to the line $R$. Here we refer to the theory of \emph{interlacing polynomials} \cite{fisk06}, briefly discussed in Section \ref{sec:Interlacing}. This theory has attracted much attention recently since Marcus, Spielman and Srivastava used it to solve the Kadison--Singer problem as well as to show the existence of bipartite Ramanujan graphs of all degrees \cite{intI,intII}. It turns out it is also very useful while studying roots of special Ehrhart polynomials.
One of the examples is a generalization of the theorem for trees presented in Lemma \ref{lem1}. Namely, consider two polynomials associated to complete bipartite graphs $(1,n)$ and $(1,n+1)$.
If we go along $R$ the roots of these two polynomials interchange, i.e.~we never encounter consecutively two roots of one polynomial. The story does not end on graphs. Indeed, other reflexive polytopes exhibit similar properties. In particular, among other families, we present proofs for classical root polytopes of type C \cite{ABHPS} and the duals of Stasheff polytopes \cite{FZ}.

Our general method consists of following steps:
\begin{enumerate}
\item Determine the Ehrhart polynomials for a family of graphs. Here we either rely on known results, or we apply a Gr\"obner degeneration of the associated toric algebra to a monomial ideal, reducing to a combinatorial problem - cf.~Proposition \ref{prop:HS}.
\item Find recursive formulas for the Hilbert series/Ehrhart polynomials. The recursive formulas may involve auxiliary polynomials.
\item Deduce the interlacing property from the recursive formulas. 
\end{enumerate}
Our main results show that:
\begin{itemize}
\item The polynomials $H_{K_n}$ associated to complete graphs have all roots on $R$ - Corollary \ref{cor:complete}. In particular, \cite[Conj.~4.8]{MHNOH} holds. 
\item $H_G$ has all roots on $R$ for any bipartite graph of type $(2,n)$ - Theorem \ref{thm:any(2,n)}. In particular, \cite[Conj.~4.7]{MHNOH} holds.
\item $H_G$ has all roots on $R$ for any complete bipartite graph of type $(3,n)$ - Theorem \ref{thm:H3n}.
\item The polynomials $H_P$ associated to classical root polytopes of type C and to dual Stasheff polytopes have all roots on $R$.
\end{itemize}
Apart from already mentioned methods like interlacing polynomials, recursive relations and Gr\"obner degenerations our research also relates to topics like orthogonal polynomial systems, hypergeometric functions and reflexive polytopes. 
We strongly believe that our techniques will be further applied to many other families of polynomials. 
\section*{Acknowledgements}
Micha{\l}ek and Kummer would like to thank the Max Planck Institute for Mathematics in the Sciences in Leipzig, in particular J{\"u}rgen Jost and Bernd Sturmfels. Micha{\l}ek also wants to thank the Research Institute for Mathematical Sciences in Kyoto, in particular Takayuki Hibi and Hiraku Nakajima, where he worked on this project.
\section{Preliminaries}\label{sec:Notation}
\subsection{Graphs, polytopes and polynomials}\label{sec:PPG}
Let us fix a graph $G=(V,E)$ and a lattice $\ZZ^{|V|}$ with basis elements $e_v$ for $v\in V$.
The associated lattice polytope $P_G\subset\RR^{|V|}$ is the convex hull of:
$$\{e_{v_1}-e_{v_2}: \{v_1,v_2\}\in E\}.$$
We note that each edge of $G$ corresponds to two vertices of $P_G$, and apart from the vertices $P_G$ contains one more lattice point $0\in \ZZ^{|V|}$. The construction above is most common when $G$ is an oriented graph. However for nonoriented graphs (or put differently for oriented graphs with the property that $(v_1,v_2)\in E\Rightarrow (v_2,v_1)\in E$) we obtain a much more symmetric situation. Recall that a lattice polytope $P$ is \emph{reflexive} if $0$ is its only interior point and the dual polytope $P^\vee$ is also integral. Furthermore, it a \emph{terminal} if every lattice point on the boundary is its vertex. For (nonoriented) graphs $G$ the polytope $P_G$ is
\emph{reflexive and terminal} \cite[Prop. 4.2]{MHNOH}
. These are very important and appear not only in combinatorics and algebra but also play a prominent role in algebraic geometry e.g.~through mirror symmetry \cite{Batyrev}. In particular, the associated toric variety is Gorenstein.

Given a lattice polytope $P\subset\RR^d$ we denote by $H_P$ its Ehrhart polynomial, i.e.~
$$H_P(s)=|sP\cap \ZZ^d|\;\;\text{ for integer }\;\; s \geq 0.$$
If $P=P_G$, to simplify notation we write $H_G$ for the associated Ehrhart polynomial. Furthermore, if $G$ is a complete $k$-partite graph of type $(a_1,\dots,a_k)\in \Z_+^k$, we write $H_{(a_1,\dots,a_k)}$. 

If $P$ is $d$-dimensional, then $H_P$ is of degree $d$ and the Ehrhart series is a rational function:
$$HS_P(t)=\frac{\sum_{i=0}^d \delta_it^i}{(1-t)^{d+1}}.$$ 
We use similar subscript notation for the Ehrhart series. 
The sequence $(\delta_0,\delta_1,\dots,\delta_d)$ is called the $\delta$-vector of $P$. Given a $\delta$-vector for $P$ the Ehrhart polynomial can be reconstructed by: $H_P(s)=\sum_{i=0}^d \delta_i{{s+d-i}\choose d}$.

We are now ready to see how the duality for reflexive polytopes reflects in the symmetry properties on the level of algebra.
\begin{prop}[cf. \cite{Batyrev, Hibidual}]\label{equiv}
Let $P$ be a lattice polytope of dimension $d$, 
$H_P(m)=a_dm^d+a_{d-1}m^{d-1}+\cdots+1$ its Ehrhart polynomial and 
$\delta(P)=(\delta_0,\delta_1,\ldots,\delta_d)$ its $\delta$-vector. 
Then the following four conditions are equivalent: 
\begin{enumerate}
\item[(a)] $P$ is a reflexive polytope; 
\item[(b)] $\delta(P)$ is palindromic, i.e., $\delta_j = \delta_{d-j}$ for $0 \leq j \leq d$;
\item[(c)] the functional equation $H_P(m) = ( - 1 )^d H_P(- m - 1)$ holds; 
\item[(d)] $d a_d=2 a_{d-1}$. 
\end{enumerate}
\end{prop}
Property (c) of the previous proposition shows exactly (skew)symmetry around $-\frac{1}{2}$. In the following we will denote $R_{a}=\{z\in\CC: \Rel(z)=a\}$. 
Our goal will be to show that polynomials we consider are not only symmetric, but have all of their roots on $R_{-\frac{1}{2}}$. Thus, we will denote $R=R_{-\frac{1}{2}}$. 
\subsection{Interlacing Polynomials}\label{sec:Interlacing}
%
The theory of interlacing polynomials has proved to be very useful in various areas of mathematics. It was a crucial ingredient for the construction of bipartite Ramanujan graphs of all degrees \cite{intI}, the solution of the Kadison--Singer problem \cite{intII} as well as the proof of the Johnson conjectures \cite{petter} to name only a few. More related to our work, recently it was also used for proving the real rootedness of polynomials appearing as the numerator of certain Hilbert series \cite{jochemko2016real}. For a comprehensive treatment of interlacing polynomials and their properties we refer to \cite{fisk06}.
\begin{df}[interlacing]
 Let $L=\alpha+\RR\cdot \beta$ be a line in $\CC$ with $\alpha,\beta\in\CC$ and let $f,g\in\CC[x]$ be univariate polynomials with $d=\deg f=\deg g+1$. Assume that both $f$ and $g$ have all their zeros on $L$. Let $\alpha+t_1\cdot \beta,\ldots,\alpha+t_d\cdot \beta$ and $\alpha+s_1\cdot \beta,\ldots,\alpha+s_{d-1}\cdot \beta$ with $t_i,s_j\in\RR$ be the zeros of $f$ and $g$ respectively. We say that $f$ is $L$-\textit{interlaced} by $g$ if 
 (after possibly relabeling) we have \[
                           t_1\leq s_1\leq t_2 \leq \cdots \leq t_{d-1}\leq s_{d-1} \leq t_d . \]
\end{df}


\begin{lema}\label{lem:interlcone}
  Let $L\subseteq\CC$ be a line, let $f,g_1,\ldots, g_r\in\CC[x]$ be monic polynomials and let $\lambda_1,\ldots,\lambda_r \in \RR$ be nonnegative real numbers. If $f$ is $L$-interlaced by each $g_i$, then $f$ is $L$-interlaced by \[\lambda_1 g_1+\ldots+\lambda_r g_r.\]
\end{lema}
\begin{proof}
 Apply an affine
 transformation that sends $L$ to $\RR$ and use \cite[Lem.~1.10]{fisk06}.
\end{proof}

\begin{lema}\label{lem:simple rec interlace}
  Let $f,g,h\in\RR[x]$ be real monic polynomials such that $\deg f=\deg g+1=\deg h+2$. Assume that there is an identity
 \[f=(x+a)\cdot g+b\cdot h\] for some $a,b\in\RR$, $b<0$. Then the following are equivalent:
 \begin{enumerate}
  \item $f$ is $\RR$-interlaced by $g$.
  \item $g$ is $\RR$-interlaced by $h$.
 \end{enumerate}
\end{lema}

\begin{proof}
 Look at the Sturm sequence associated to $f$ and $g$, cf. \cite[\S 2.2.2]{sturm}.
\end{proof}

\begin{lema}\label{lem:interl}
  Let $f_1,f_2,f_3\in\RR[x]$ be real monic polynomials such that $\deg f_1=\deg f_2+1=\deg f_3+2$. Assume that there is an identity
 \[f_1=(x+a)\cdot f_2+b\cdot f_3\] for some $a,b\in\RR$, $b>0$. Furthermore, let \[(-1)^{\deg f_i} f_i(x)=f_i(2d-x)\] for some $d\in\RR$ and $i=1,2,3$. Then the following are equivalent:
 \begin{enumerate}
  \item $f_1$ is $R_d$-interlaced by $f_2$.
  \item $f_2$ is $R_d$-interlaced by $f_3$.
 \end{enumerate}
 If $(i)$ and $(ii)$ are satisfied, then $(x-d)f_3$ $R_d$-interlaces $f_1$.
\end{lema}

\begin{proof}
 First note that we necessarily have $a=-d$. If we replace $x$ by $\textnormal{i}x+d$ and divide by $\textnormal{i}^{\deg f_1}$, we get the first claim from the preceding lemma. In order to prove the second statement, we compute the next elements of the Sturm sequence:
 \[
  f_2=(x-d)f_3+b' f_4,
 \]
\[
 f_3=(x-d)f_4+b'' f_5,
\]
for some monic $f_4,f_5\in\RR[x]$ of degree $\deg f_1-3$ and $\deg f_1-4$ respectively and $b',b''>0$. From this we get \[f_1(x+d)=(x^2+b+b') f_3(x+d)-b' b'' f_5(x+d).\]
After possibly dividing by $x$ we can assume that the polynomials $f_1(x+d)$, $f_3(x+d)$ and $f_5(x+d)$ are even. Replacing each occurence of $x^2$ by $x$ we get \[g_1=(x+b+b')g_3-b'b'' g_5\]
for some polynomials $g_i$ with $g_i(x^2)=f_i(x+d)$. Note that the $g_i$ have only real and nonpositive roots. Thus, by \cite[Lem. 1.82]{fisk06} $g_3$ interlaces $g_1$. This implies the claim.
\end{proof}
\subsection{Orthogonal Polynomial Sequences}\label{sec:ops}
The theory of orthogonal polynomials is a classical topic in mathematics. 
\begin{df}[positive-definite moment functional, orthogonal polynomial system]
A linear function $\mu:\CC[x]\rightarrow \CC$ is called a \emph{moment functional}. A moment functional $\mu$ is called \emph{positive-definite} if for any nonzero polynomial $f$ that takes nonnegative values on $\RR$ we have $\mu(f)\in\RR$ and $\mu(f)>0$.  
Let us fix a moment functional $\mu$.
A sequence of polynomials $\{f_d\}_{d \in \N}$, where $\deg f_d=d$ is called \emph{an orthogonal polynomial system} (OPS) if 
$\mu(f_df_e)=0$ if $d \neq e$ and $\mu(f_df_e)\neq 0$ if $d=e$ for all $d,e\in\N$.
\end{df}
%
%

Examples for OPS include the Hermite polynomials, the Laguerre polynomials and the Jacobi polynomials. 

\begin{thm}[Favard's Theorem, \cite{Chihara} Ch.~I, Thm.~4.1, Thm.~4.4]\label{thm:Favard}
Let $(c_j)_{j \in \N}$, $(\lambda_j)_{j \in \N}$ be arbitrary sequences of complex numbers. Let $(f_j(x))_{j=-1}^\infty$ be a sequence of polynomials defined by:
$$f_{-1}(x)=0, f_0(x)=1$$
$$f_j(x)=(x-c_j)f_{j-1}(x)-\lambda_jf_{j-2}(x), \text{ for }j=1,\dots.$$ 
Then there exists a positive-definite moment functional $\mu$ with respect to which $f_j$ is an OPS if and only if $c_j$ is real and $\lambda_j>0$ for each $j$. 
Further, for any positive-definite moment functional $\mu$ there exists a unique monic OPS and it satisfies the recurrence relation above.
\end{thm}

Note that in this case Lemma \ref{lem:simple rec interlace} implies that every $f_j$ has only real roots and is interlaced by $f_{j-1}$. This is also the content of the Separation Theorem, cf. \cite[Thm.~I.5.3]{Chihara}. This property was used in \cite{BCKV} to show that certain polynomials obtained from the Mellin transform of Hermite polynomials and polynomials $H_{1,n}(-x)$ obtained from the Mellin transform of Laguerre polynomials
have all their roots on the line $R_{\frac{1}{2}}$.
 
\section{Ehrhart Polynomials of Reflexive Polytopes as Orthogonal Polynomial Sequences}
 In this section we will be interested in sequences of reflexive polytopes $(P_j)_{j=0}^\infty$ whose Ehrhart polynomials $f_j(x)=H_{P_j}(x)$ satisfy a recurrence relation $$f_j=(a_j x+b_j)f_{j-1}+c_j f_{j-2}, \text{ for }j=2,3,\dots$$ for some $a_j,b_j,c_j\in\QQ$. Note that by Proposition \ref{equiv} and since the constant term of every Ehrhart polynomial is equal to one we can write the relation as $$f_j=M_j(2x+1)f_{j-1}+(1-M_j) f_{j-2}, \text{ for }j=2,3,\dots$$ for some $M_j\in\QQ$. If $0\leq M_j\leq1$, then it follows from Lemma \ref{lem:interl} that every $f_j$ has all its roots on $R$ and is $R$-interlaced by $f_{j-1}$ and $(x+\frac{1}{2})f_{j-2}$. Of course, the modified sequence $(\tilde{f}_j)_{j=-1}^\infty$ where $\tilde{f}_j(x)=(-\textnormal{i})^d f(\textnormal{i}x-\frac{1}{2})$ is in that case an OPS.
 
 Note that since it is well-known that $2x+1$ is the unique Ehrhart polynomial of a reflexive polytope of dimension one and that every two dimensional reflexive polytope has the Ehrhart polynomial $ax^2+ax+1$, where $a \in \{\frac{i}{2} : i =3,4,\ldots,9\}$, we can conclude that $M_2$ should be one of $\frac{3}{8},\frac{4}{8},\ldots,\frac{9}{8}$. In the following examples we will see that the values $\frac{4}{8},\frac{5}{8},\frac{6}{8},\frac{8}{8}$ actually appear. Furthermore, as we will see the polytopes giving rise to these OPS come in interesting families appearing in different branches of mathematics. 

\begin{con} There do not exist families of OPS coming from Ehrhart polynomials, such that $M_2=\frac{3}{8}$ or $M_2=\frac{7}{8}$.
\end{con}
\begin{rem}
We can confirm the previous conjecture under the assumption that $M_n$ is a decreasing rational function of $n$.
\end{rem}

\begin{exm}[Cross Polytope]\label{cross}
Let $\mathrm{Cr}_d$ be the convex hull of $\{ \pm \eb_i : 1 \leq i \leq d\}$. 
Then $\mathrm{Cr}_d$ is a reflexive polytope of dimension $d$, called the {\em cross polytope}. 
Its Ehrhart polynomial can be computed as follows: 
$$H_{\mathrm{Cr}_d}(m)=\sum_{k=0}^d \binom{d}{k}\binom{m+d-k}{d}.$$ 
Moreover, we see that $\{H_{\mathrm{Cr}_d}(m)\}_{d=0}^\infty$ satisfies the recurrence relation
\begin{align}\label{crossrec}
H_{\mathrm{Cr}_d}(m)=\frac{1}{d}(2m+1)H_{\mathrm{Cr}_{d-1}}(m)+\frac{d-1}{d}H_{\mathrm{Cr}_{d-2}}(m)\;\text{ for all }d \geq 2. 
\end{align}In fact, the direct computations show the following: 
\begin{align*}
&\sum_{m=0}^\infty\left(\frac{2m+1}{d}\sum_{k=0}^{d-1} \binom{d-1}{k}\binom{m+d-1-k}{d-1}
+\frac{d-1}{d}\sum_{k=0}^{d-2} \binom{d-2}{k}\binom{m+d-2-k}{d-2}\right)t^m \\
&=\frac{2t}{d}\cdot\left(\frac{(1+t)^{d-1}}{(1-t)^d}\right)'+\frac{1}{d}\cdot\frac{(1+t)^{d-1}}{(1-t)^d}
+\frac{d-1}{d}\cdot\frac{(1+t)^{d-2}}{(1-t)^{d-1}} =\frac{(1+t)^d}{(1-t)^{d+1}}\\
&=\sum_{m=0}^\infty \left(\sum_{k=0}^d \binom{d}{k}\binom{m+d-k}{d}\right)t^m.
\end{align*}
Therefore, the Ehrhart polynomial of the cross polytope $\mathrm{Cr}_d$ has all roots on $R$ and $H_{\mathrm{Cr}_{d+1}}$ is $R$-interlaced by $H_{\mathrm{Cr}_d}$. 

Note that since the zeros of the polynomial $\sum_{k=0}^d \binom{d}{k}t^k=(1+t)^k$ in $t$ 
are all $-1$, we can also prove that $H_{\mathrm{Cr}_d}(m)$ has the roots on $R$ by applying \cite{rod02} as mentioned in the introduction. 
Moreover, the $H_{\mathrm{Cr}_d}(m)$ coincides with $H_G$ for any tree $G$ with $(d+1)$ vertices. 
\end{exm}

\begin{exm}[Dual of the Stasheff Polytope]\label{associahedra}
Let $\mathrm{St}_d$ be the convex hull of $\{ \pm \eb_i : 1 \leq i \leq d\} \cup \{\eb_i+\cdots+\eb_j : 1 \leq i < j \leq d\}$. 
Then $\mathrm{St}_d$ is a reflexive polytope of dimension $d$. 
This polytope is the dual polytope of the so-called {\em Stasheff polytope} ({\em associahedron}). 
For more detailed information, see, e.g., \cite{FZ}. In \cite{Ath}, its Ehrhart polynomial is calculated as follows: 
\begin{align*}
H_{\mathrm{St}_d}(m)=\sum_{k=0}^d \frac{1}{d+1}\binom{d+1}{k+1}\binom{d+1}{k}\binom{m+d-k}{d}.
\end{align*}
Similar to Example \ref{cross}, it follows from the direct computations that $\{H_{\mathrm{St}_d}(m)\}_{d=0}^\infty$ satisfies 
the recurrence relation
$$H_{\mathrm{St}_d}(m)=\frac{2d+1}{d(d+2)}(2m+1)H_{\mathrm{St}_{d-1}}(m)+\frac{(d-1)(d+1)}{d(d+2)}H_{\mathrm{St}_{d-2}}(m)\;\text{ for all }d \geq 2.$$
Therefore, $H_{\mathrm{St}_d}$ has all roots on $R$ 
and $H_{\mathrm{St}_{d+1}}$ is $R$-interlaced by $H_{\mathrm{St}_d}$. 
\end{exm}

\begin{exm}[Classical Root Polytope of Type A]\label{typeA}
Let ${\bf A}_d$ be the convex hull of the root system of type A, i.e., 
$\{\pm \eb_i : 1 \leq i \leq d\} \cup \{ \pm (\eb_i+\cdots+\eb_j) : 1 \leq i < j \leq d\}$. 
Then ${\bf A}_d$ is a reflexive polytope of dimension $d$, called the {\em classical root polytope of type A}. 
For more detailed information, see \cite{ABHPS}.
Note that the definitions of this paper and \cite{ABHPS} look different, but these are unimodularly equivalent. 
Its Ehrhart polynomial is calculated in \cite[Thm. 1]{BDV} and also in \cite[Thm. 2]{ABHPS} as follows: $$H_{{\bf A}_d}(m)=\sum_{k=0}^d \binom{d}{k}^2\binom{m+d-k}{d}.$$ 
It follows that $\{H_{{\bf A}_d}(m)\}_{d=0}^\infty$ satisfies the recurrence relation
$$H_{{\bf A}_d}(m)=\frac{2d-1}{d^2}(2m+1)H_{{\bf A}_{d-1}}(m)+\frac{(d-1)^2}{d^2}H_{{\bf A}_{d-2}}(m)\;\text{ for all }d \geq 2.$$
Therefore, $H_{{\bf A}_d}$ has all roots on $R$ and $H_{{\bf A}_{d+1}}$ is $R$-interlaced by $H_{{\bf A}_d}$. 
\end{exm}

\begin{exm}[Classical Root Polytope of Type C]\label{typeC}
Let ${\bf C}_d$ be the convex hull of the root system of type C, i.e., 
$\{\pm \eb_i : 1 \leq i \leq d \} \cup \{ \pm (\eb_i+\cdots+\eb_{j-1}) : 1 \leq i < j \leq d \} \cup \{\pm(2\eb_i+\cdots+2\eb_{d-1}+\eb_d) : 1 \leq i \leq d-1\}$. 
Then ${\bf C}_d$ is a reflexive polytope of dimension $d$, called the {\em classical root polytope of type C}. 
For more detailed information, see \cite{ABHPS}. Its Ehrhart polynomial is calculated in \cite[Thm. 1]{BDV} and also in \cite[Thm. 2]{ABHPS} as follows: 
$$H_{{\bf C}_d}(m)=\sum_{k=0}^d \binom{2d}{2k}\binom{m+d-k}{d}.$$ It follows that $\{H_{{\bf C}_d}(m)\}_{d=0}^\infty$ satisfies 
the recurrence relation $$H_{{\bf C}_d}(m)=\frac{2}{d}(2m+1)H_{{\bf C}_{d-1}}(m)+\frac{d-2}{d}H_{{\bf C}_{d-2}}(m)\;\text{ for all }d \geq 2.$$
We conclude that $H_{{\bf C}_d}$ has all roots on $R$ and $H_{{\bf C}_{d+1}}$ is $R$-interlaced by $H_{{\bf C}_{d}}$. 
\end{exm}

\begin{rem}
Let ${\bf B}_d$ (resp. ${\bf D}_d$) be the classical root polytope of type B (resp. D) of dimension $d$. 
In \cite{BDV}, ${\bf B}_d$ and ${\bf D}_d$ are also discussed. It is proved in \cite[Thm. 1]{BDV} that 
\begin{align*}
HS_{{\bf B}_d}(t)=\frac{\sum_{k=0}^d\binom{2d+1}{2k}t^k-2dt(1+t)^{d-1}}{(1-t)^{d+1}} \;\text{ and }\;
HS_{{\bf D}_d}(t)=\frac{\sum_{k=0}^d\binom{2d}{2k}t^k-2dt(1+t)^{d-2}}{(1-t)^{d+1}}. 
\end{align*}
We see that the $\delta$-vector of ${\bf B}_d$ is not palindromic, so the roots of $H_{{\bf B}_d}$ are not distributed symmetrically with respect to $R$. The $\delta$-vector of ${\bf D}_d$ is palindromic, so the roots of $H_{{\bf D}_d}$ are distributed symmetrically with respect to $R$. However, they do not have to lie on $R$.
Below we present how the roots of $H_{{\bf B}_6}$ (on the left) and $H_{{\bf D}_6}$ (on the right) are distributed.
\adjustbox{trim={.0\width} {.5\height} {0.0\width} {.0\height},clip}%
{\includegraphics[scale=0.45]{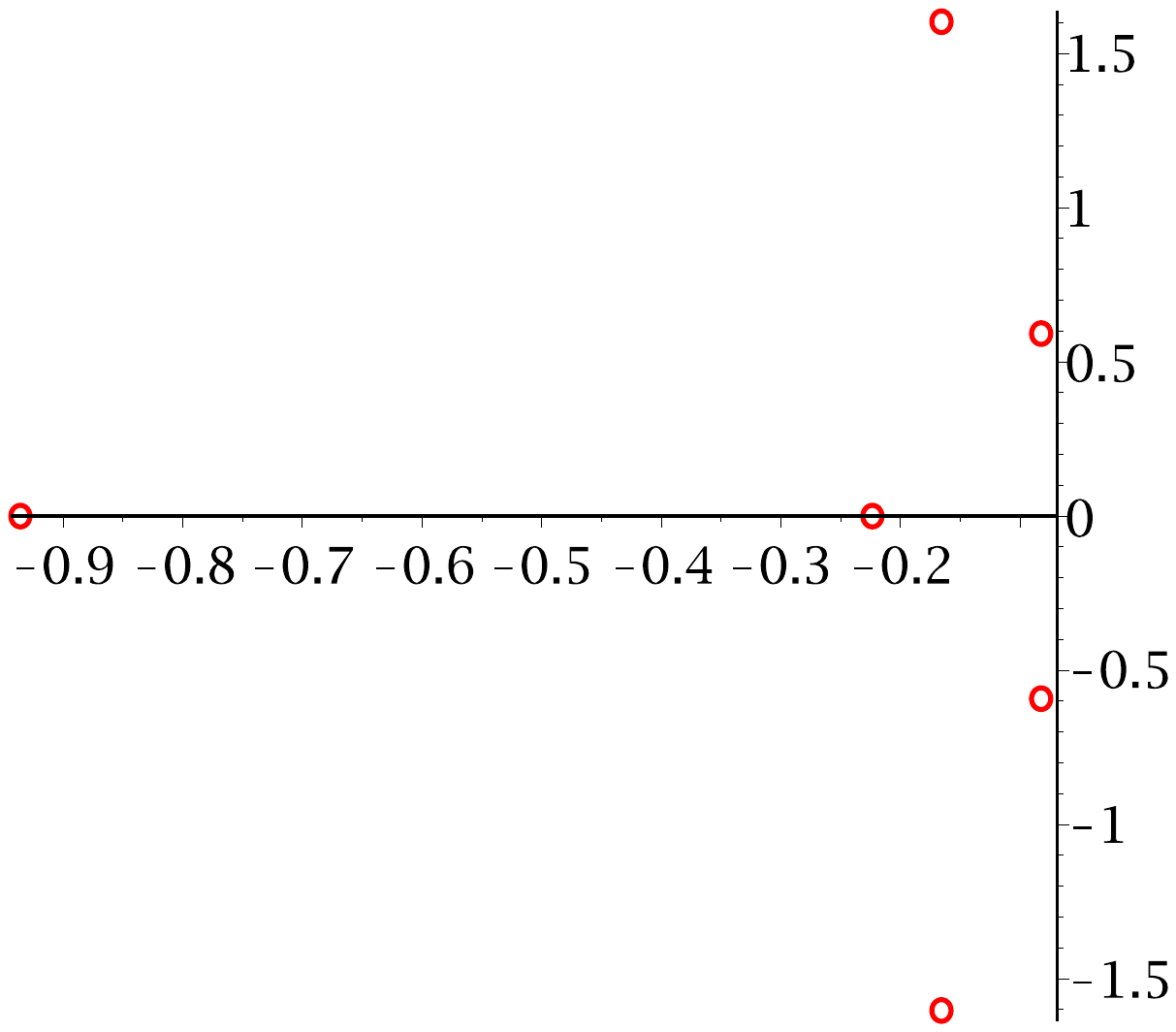}
\includegraphics[scale=0.45]{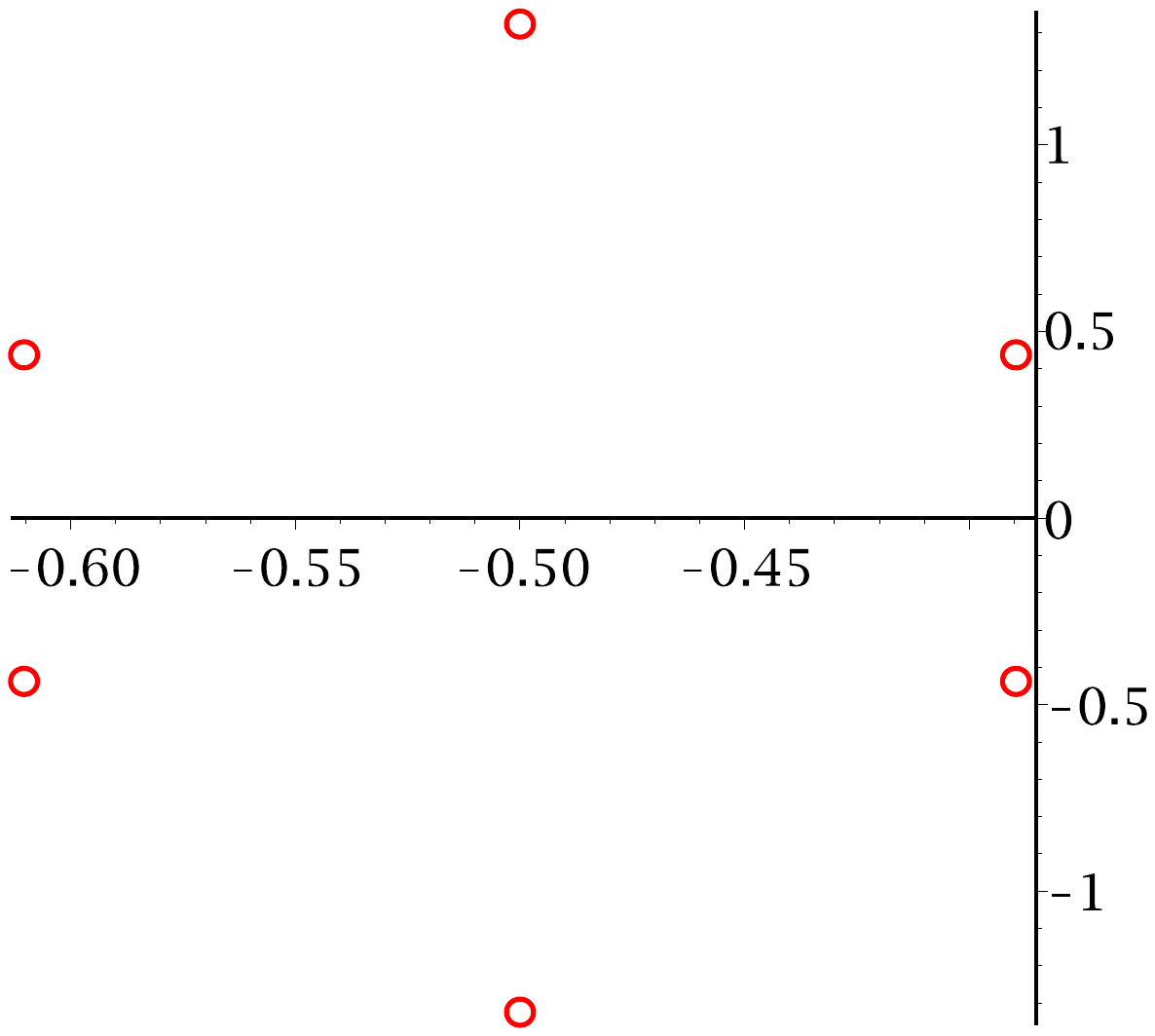}}
\end{rem}
As a consequence of Example \ref{typeA}, we also obtain the following: 
\begin{cor}[{cf. \cite[Conj. 4.8]{MHNOH}}]\label{cor:complete}
For any complete graph $K_n$, $H_{K_n}$ has all roots on $R$, i.e., \cite[Conj. 4.8]{MHNOH} is true. 
\end{cor}
\begin{proof}
The classical root polytope of type A, ${\bf A}_d$, is unimodularly equivalent to the symmetric edge polytope of complete graphs. 
Hence, Example \ref{typeA} directly proves the assertion. 
\end{proof}

\begin{rem}
Orthogonal polynomial systems are also studied in relation to hypergeometric functions \cite{koekoek2010hypergeometric}. We note that the Ehrhart polynomials described in Examples \ref{cross}, \ref{associahedra}, \ref{typeA} and \ref{typeC}  can be presented also in that language:
\begin{align*}
&H_{\mathrm{Cr}_d}(m)=\binom{m+d}{m} \;_2F_1(-d,-m; -d-m; 1),\\
&H_{\mathrm{St}_d}(m)=\binom{m+d}{m} \;_3F_2(-d-1,-d,-m; 2, -d-m; 1),\\
&H_{{\bf A}_d}(m)=\binom{m+d}{m} \;_3F_2(-d,-d,-m; 1,-d-m; 1),\\
&H_{{\bf C}_d}(m)=\binom{m+d}{m} \;_5F_4(1/2-d,1/2-d,-d,-d,-m; 1/2,1/2,1,-d-m; 1), 
\end{align*}
where ${}_rF_s$ denotes the hypergeometric function. 
\end{rem}

\section{Ehrhart polynomials and bipartite graphs}
As we have seen in the previous section OPS provide strong methods to prove that zeros of polynomials lie on $R$. However, as conjectured e.g.~in \cite[Conj.~4.7]{MHNOH} there exist families with roots on $R$, but not giving rise to OPS. Indeed, one can check that the polynomials $H_{2,n}$ do not satisfy the recurrence relations in Theorem \ref{thm:Favard}. As we shall see, for this setting, the correct generalization of orthogonal polynomial systems are interlacing polynomials - cf.~Lemmas \ref{lem1}, \ref{lem2}, \ref{lem4}, Theorem \ref{thm:H3n} and Conjecture \ref{con:main}.
\subsection{Complete bipartite graphs of type $(2,n)$ and $(3,n)$}
In this section we show that the Ehrhart polynomials $H_{2,n}$ and $H_{3,n}$ of the symmetric edge polytope of the complete bipartite graph of type $(2,n)$ and $(3,n)$ respectively has all of its roots on the line $R=R_{-\frac{1}{2}}$.

First, we determine the corresponding Hilbert series.
We start by degenerating the binomial ideal to the monomial one.
For any complete bipartite graph $K_{a,b}$ let us order the vertices in each part $v_1<\dots<v_a$, $w_1<\dots<w_b$ and further $v_i<w_j$. There is an induced order on edges: $(v_i,w_j)<(v_{i'},w_{j'})$ if $i<i'$ or $i=i'$ and $j<j'$, analogously for the other orientation and $(v_i,w_j)<(w_{i'},v_{j'})$ for any $i,j,i',j'$. Further we declare the unique interior point $0$ of the associated polytope $P_{a,b}$ to be smaller than any point corresponding to the edges.   
\begin{lema}\label{lem:QGB}
The ideal $I_{a,b}$ with respect to the induced degrevlex order has a quadratic Gr\"obner basis corresponding to:
\begin{enumerate}
\item $(v_i,w_j)(w_j,v_i)-2\cdot 0$,
\item $(v_i,w_j)(w_j,v_{i'})-(v_i,w_1)(w_1,v_{i'})$,
\item $(w_i,v_j)(v_j,w_{i'})-(w_i,v_1)(v_1,w_{i'})$,
\item $(v_i,w_j)(v_{i'},w_{j'})-(v_i,w_{j'})(v_{i'},w_{j})$ for $i>i'$ and $j<j'$,
\item $(w_i,v_j)(w_{i'},v_{j'})-(w_i,v_{j'})(w_{i'},v_{j})$ for $i>i'$ and $j<j'$.
\end{enumerate}
The leading terms are presented on the left. 
\end{lema}
\begin{proof}
We leave the easy proof as the exercise for the reader. A more general statement can be found in \cite{ohsugi2014centrally}.
\end{proof}
The previous lemma motivates the following definition.
\begin{df}[correct graph, $f(a,b,k)$]
A directed bipartite graph without any subgraphs corresponding to leading terms in the Lemma \ref{lem:QGB} is called \emph{correct}. Let $f(a,b,k)$ be the number of correct $(a,b)$-bipartite graphs with exactly $k$ edges (counted with multiplicities).
\end{df}
By the standard degeneration argument to the initial ideal we obtain:
\begin{cor}\label{cor:graphstobecounted}
The Ehrhart polynomial $H_{a,b}$ on value $k$ counts the number of $(a,b)$-bipartite correct graphs with at most $k$ edges (possibly repeated).
\end{cor}
Despite this very explicit combinatorial description, in general we do not know how to determine the Ehrhart polynomial. Still, in specific situations we may find combinatorial recursive relations and easily prove a given formula satisfies them.
\begin{prop}\label{prop:HS}
The Hilbert series for the complete graph bipartite $K_{1,n}$ equals
$$HS_{1,n}(t)=\frac{(1+t)^{n}}{(1-t)^{n+1}}.$$

The Hilbert series for the complete bipartite graph $K_{2,n}$ equals
$$HS_{2,n}(t)=\frac{(1+t)^{n-1}(1+2nt+t^2)}{(1-t)^{n+2}}.$$

The Hilbert series for the complete bipartite graph $K_{3,n}$ equals $$HS_{3,n}(t)=\frac{(1+t)^{n-2}(1+4nt+(3n^2-n+4)t^2+4nt^3+t^4)}{(1-t)^{n+3}}.$$
\end{prop}
\begin{proof}
The first part is well-known, cf.~Example \ref{cross}. The second was stated in \cite{MHNOH}.

The proof of the third has two steps. In the first we determine a recursive relation the Ehrhart polynomial must satisfy. In the second we deduce the Hilbert series form it.

1) Our aim is to find a formula for $f(3,n,k)$. We note that in the graphs we count, any vertex apart from $v_1$ and $w_1$ is either an in or out vertex. To simplify the terminology (e.g.~degree of the vertex) we consider the graphs we count as \emph{simple} and directed edges as (positively) weighted. 

There are $f(3,n-1,k)$ graphs for which $v_n$ is of degree 0.

Consider a correct graph $G$. By removing $v_n$ and edges adjacent to it we obtain a $(3,n-1)$-bipartite correct graph $\tilde G$. Let us introduce the notation to count different types of correct graphs according to whether $w_3$ and $w_2$ are outgoing or incoming. We denote by $f_{xy}(3,n,j)$ where $x,y\in\{i,o,z\}$ the number of correct $(3,n)$ bipartite graphs with $j$ edges such that $w_3$ is of type $x$ (i.e.~\emph{i}ncoming, \emph{o}utgoing or of degree \emph{z}ero) and $w_2$ is of type $y$. From now on we count graphs for which $v_n$ is not of degree zero and as both cases are similar we assume it is outgoing. According to the type of $\tilde G$ we sum up:
\begin{enumerate}
\item $\sum_{j=0}^{k-1}f_i(3,n-1,j)$. Here we count graphs for which $w_3$ is incoming in $\tilde G$ - note that in this case the only possibility to obtain $G$ is to add the edge $(v_n,w_3)$ (with multiplicity $k-j$),
\item $\sum_{j=0}^{k-1}(k-j+1)f_{zi}(3,n-1,j)$. Here we multiply by $(k-j+1)$ as we can distribute the weight $(k-j)$ among the edges $(v_n,w_3)$ and $(v_n,w_2)$.
\item $\sum_{j=0}^{k-1}f_{oi}(3,n-1,j)$
\item $\sum_{j=0}^{k-1}{{k-j+2}\choose 2}f_{zz}(3,n-1,j)$. Here we distribute among all three edges.
\item $\sum_{j=0}^{k-1}(k-j+1)f_{zo}(3,n-1,j)$
\item $\sum_{j=0}^{k-1}f_{oo}(3,n-1,j)$
\item $\sum_{j=0}^{k-1}(k-j+1)f_{oz}(3,n-1,j)$
\end{enumerate}
We obtain the same when $v_n$ is incoming, apart from the fact that we have to repalce $i$ and $o$ in all formulas. Summing up all we get:
$$f(3,n,k)=f(3,n-1,k)+\sum_{j=0}^{k-1}(2f(3,n-1,j)+3(k-j)f(2,n-1,j)+(k-j)^2f(1,n-1,j)).$$



2) To pass from the recursive relation to the Hilbert series we proceed as follows. First we note that the Ehrhart polynomial $H_{3,n}(i)=\sum_{k=0}^i f(3,n,k)$ and equivalently $f(3,n,j)=H_{3,n}(j)-H_{3,n}(j-1)$. Thus summing up the recursive relation we obtain:
\begin{eqnarray}\label{recformforH3n}\nonumber H_{3,n}(i)&=&H_{3,n-1}(i)+\sum_{0\leq j<k\leq i}(2f(3,n-1,j)+3(k-j)f(2,n-1,j)+(k-j)^2f(1,n-1,j))\nonumber\\&=&
-H_{3,n-1}(i)+\sum_{k=0}^{i} 2H_{3,n-1}(k)+\\ \nonumber
&\sum_{j=0}^i&(\frac{3}{2}(i-j+1)(i-j)f(2,n-1,j)+\frac{(i-j)(i-j+1)(2i-2j+1)}{6}f(1,n-1,j))
\end{eqnarray}
At this point we could substitute all the values on the right hand side and conclude. This however involves a lot of nontrivial computation on binomial coefficients. A better way is to pass to the Hilbert series. Precisely we multiply both sides of the equality (\ref{recformforH3n}) by $t^i$ and sum up over all natural $i$ obtaining:
\begin{eqnarray}\label{recformforHS3n} HS_{3,n}(t)&=&-HS_{3,n-1}(t)+2\frac{HS_{3,n-1}(t)}{1-t}+\frac{3tHS_{2,n-1}(t)}{(1-t)^2}+\frac{(t^2+t)HS_{1,n-1}(t)}{6(1-t)^3}\\
\nonumber &=&\frac{(1+t)HS_{3,n-1}(t)}{1-t}+\frac{3tHS_{2,n-1}(t)}{(1-t)^2}+\frac{(t^2+t)HS_{1,n-1}(t)}{(1-t)^3}
\end{eqnarray}
Here the four Hilbert series in (\ref{recformforHS3n}) (in order) correspond exactly to four terms in the recursive relation (\ref{recformforH3n}). Now the claim of the proposition is reduced to verifying that:
\begin{eqnarray*}
1+4nt+(3n^3-n+4)t^2+4nt^3+t^4=\\1+4(n-1)t+(3(n-1)^2-(n-1)+4)t^2+4(n-1)t^3+t^4\\+
3t(1+2(n-1)t+t^2)+(t^2+t)(1+t)\end{eqnarray*}
\end{proof}
So far we have used recursive relations to determine Hilbert series. However, it is also useful to go the other way round and determine further recursive relations using the known formulas. 
\begin{prop}\label{prop:relations}
The following relations hold:
\begin{eqnarray}
   H_{2,n}(k)&=&\frac{1}{2} (2k+1) H_{1,n}(k) + \frac{1}{2} H_{1,n-1}(k),\label{relH2n1}\\
H_{2,n}(k)&=&\frac{1}{n}(2k+1) H_{2,n-1}(k)+\frac{1}{2n}(n H_{1,n-1}(k)+(n-2)(2k+1)H_{1,n-2}(k)),\label{relH2n2}\\
H_{3,n+1}(k) &=& \left( \frac{3 n^2+13 n+16}{4 \left(n^2+5 n+6\right)}k +  \frac{3 n^2+13 n+16}{8 \left(n^2+5 n+6\right)}\right)H_{2,n+1}(k)\\\nonumber &+&  \frac{n^3+13 n^2+18 n}{8 (n-1) \left(n^2+5 n+6\right)}H_{2,n}(k) + \frac{4 n^3+9n^2-13 n-32}{8 (n-1) \left(n^2+5 n+6\right)}H_{1,n+1}(k)\label{recH3n}\end{eqnarray}
\end{prop}
\begin{proof}
Looking at Ehrhart polynomials it may be surprising that any relations of this sort hold, as the system of equalities one gets seems overdetermined for large $n$. However, after passing to the Hilbert series we see that we only need to present one low degree polynomial as a linear combination of other polynomials. Instead of doing computation by hand one can use Mathematica \cite{Mathematica}.
First we check the relations \ref{relH2n1} and \ref{relH2n2} by subtracting the left hand side from the right hand side. Using simple algebra relations among polynomials and Hilbert series, such as the fact that $kH(k)$ has Hilbert series $t(\frac{d(HS(t))}{dt})$ we verify that:
\begin{verbatim}
In: H1[k_,t_]:=((1+t)^k)/((1-t)^(k+1))
In: H2[k_,t_]:=(1+t)^(k-1)*(1+2k*t+t^2)/(1-t)^(k+2)
In: H3[k_,t_]:=(1+t)^(k-2)*(1+4k*t+(3*k^2-k+4)*t^2+4k*t^3+t^4)/(1-t)^(k+3)
In: Simplify[2*H2[k,t]-(2*t*D[H1[k,t],t]+H1[k,t]+H1[k- 1,t])]
Out: 0
In: Simplify[k*H2[k,t]-(2*t*D[H2[k-1,t],t]+H2[k-1,t]+(k/2)*H1[k-1,t]+
((k-2)/2)*(2*t*D[H1[k-2,t],t]+H1[k-2,t]))]
Out: 0
In: Simplify[H3[k+1,t]-((16+13*k+3*k^2)/(4*(6+5*k+k^2))*t*D[H2[k+1,t],t]+
(16+13*k+3*k^2)/(8*(6+5*k+k^2))*H2[k+1,t]+(18*k+13*k^2+k^3)/(8*(k-1)*(6+5*k+k^2))*
H2[k,t]+(-32-13*k+9*k^2+4*k^3)/(8*(k-1)*(6+5*k+k^2))*H1[k+1,t])]
Out: 0
\end{verbatim}

In Appendix \ref{Appendix} we show how to not only check the relations, but determine the coefficients without knowing them.
\end{proof}

\begin{lema}\label{lem1}
 $H_{1,n}$ has all its roots on $R$ and $R$-interlaces $H_{1,n+1}$.
\end{lema}
\begin{proof}
 This was shown in Example \ref{cross}.
\end{proof}
\begin{lema}\label{lem2}
 $H_{2,n}$ has all its roots on $R$ and is $R$-interlaced by $H_{1,n}$ and $(2k+1)H_{1,n-1}$.
\end{lema}
\begin{proof}
 This follows from the relation (\ref{relH2n1}) and Lemma \ref{lem:interl}.
 \end{proof}

\begin{lema}\label{lem4}
  $H_{2,n}$ interlaces $H_{2,n+1}$.
\end{lema}
\begin{proof}
 Since both $H_{1,n-1}$ and $(2k+1)H_{1,n-2}$ interlace $H_{2,n-1}$, the claim follows from Lemma \ref{lem:interlcone}, Lemma \ref{lem:interl} and relation (\ref{relH2n2}).
\end{proof}

\begin{thm}\label{thm:H3n}
 $H_{3,n}$ has all its roots on $R$ and is interlaced by $H_{2,n}$.
\end{thm}

\begin{proof}
 $H_{2,n+1}$ is interlaced by $H_{2,n}$ and $H_{1,n+1}$, so the claim follows from relation (\ref{recH3n}).
\end{proof}

We end this section with the following conjecture which encapsulates many of our results.

\begin{con}\label{con:main}\
\begin{enumerate}
\item For any complete $k$-partite graph $G$ of type $(a_1,\dots,a_k)$ the Ehrhart polynomial $H_{a_1,\dots,a_k}$ has roots on $R$.
\item Suppose $a_1\geq\dots\geq a_k$. Any two Ehrhart polynomials $H_{a_1,\dots,a_k}$, $H_{a_1-1,a_2,\dots,a_k}$ $R$-interlace. 
\end{enumerate}
\end{con}
Our results confirm the conjecture for $a_1=\dots=a_k=1$ and also $k=2$, $a_2=1,2$.
Furthermore, numerical experiments suggest that $H_{a_1,\dots,a_k}$, $H_{a_1-1,a_2,\dots,a_k}$ $R$-interlace whenever $a_1\geq 2$. 
We have shown this to be true in the cases $k=2$ and $a_1=1,2,3$ and checked for all graphs with at most $10$ vertices.
\begin{exm}
$$H_{3,3}(x)=(9/10) x^{5}+(9/4) x^{4}+(16/3) x^{3}+(23/4) x^{2}+(113/30) x+1$$
$$H_{3,3,1}(x)=(49/60) x^{6}+(49/20) x^{5}+(37/6) x^{4}+(33/4) x^{3}+(481/60) x^{2}+(43/10) x+1$$
Their roots are approximately respectively:
$$\{-.5-1.7292i\}, \{-.5-.6602i\}, \{-.5\}, \{-.5+.6602i\},\{-.5+1.7292i\}$$
$$ \{-.5-1.6154i\},\{-.5-1.0638i\}, \{-.5-.2448i\}, \{-.5+.2448i\}, \{-.5+1.0638i\}, \{-.5+1.6154i\}$$ 
 and do not $R$-interlace.
\end{exm}
\subsection{Bipartite graphs of type $(2,n)$}
By a more carfully study of the involved polynomials we will show in this section that in fact the Ehrhart polynomial of every (not just complete) bipartite graph of type $(2,n)$ has all roots on $R$.
It will be derived as a special case for a more general statement for a larger family of polynomials. 
\begin{thm}\label{thm:any(2,n)}
Let $G$ be a bipartite graph of type $(2,n)$. Then all roots of $H_G$ belong to $R$.
\end{thm}

For all natural numbers $j<d$ there exists a polynomial $H^d_j\in\RR[x]$ of degree $d-1$ such that if $\frac{(1+t)^j}{(1-t)^d}=\sum_{k=0}^\infty h_k t^k$, then $h_k=H^d_j(k)$ for all $k\geq 0$. It was shown in \cite{rod02} that one has \[H^d_j=(x+1)\cdots(x+d-1-j)\cdot \tilde{H}^d_j\] for some polynomial $\tilde{H}^d_j\in\RR[x]$ of degree $j$ all of whose roots $\alpha\in\CC$ satisfy $\Rel(\alpha)=-\frac{d-j}{2}$.

\begin{rem}
 We have: \begin{align}\label{eqn:hypgeo} H^d_j(x)=\sum_{i=0}^j \binom{j}{i} \binom{x+d-1-i}{d-1}=\binom{x+d-1}{d-1}\cdot {}_2F_1(-j,-x;1-d-x;-1)\end{align}
\end{rem}

\begin{rem}
 The polynomials $H^d_j$ satisfy $(-1)^{d-1}H^d_j(x)=H^d_j(-d+j-x)$.
\end{rem}

\begin{lema}\label{lem:recpol}
 The polynomials $F_j^d(x)=\tilde{H}^d_j(x-\frac{d-j}{2})$ satisfy the recursion
 \[
  F_{j+2}^d=(4x^2+2 d j+d-2 j^2-3 j-2)\cdot F_j^d+j(j-1)(4x^2-(d-j)^2)\cdot F_{j-2}^d
 \]
 for $j\geq2$. Furthermore, one has $F^d_0=\frac{1}{(d-1)!}$ and $F^d_1=\frac{2x}{(d-1)!}$.
\end{lema}
\begin{proof}
 The recursion above is equivalent to the following recursion on the $H^d_j$:
\begin{equation}\label{eq:rec}
  a_j H^d_{j+2}(x+1)=b_j H^d_j(x)+c_j H^d_{j-2}(x-1)
\end{equation}
where 
\begin{align*}
a_j&=x(x+1)(x+d-j-1)(x+d-j),\\
b_j&=x(x+d-j)(4(x+\frac{d-j}{2})^2+2 d j+d-2 j^2-3 j-2),\\
c_j&=(4(x+\frac{d-j}{2})^2-(d-j)^2).
\end{align*}
Applying the formula for the Hilbert series of $H^d_j$ one can show that both sides of Equation (\ref{eq:rec}) have the same Hilbert series.
\end{proof}


It follows from Lemma \ref{lem:recpol} that $F_j^d$ is an even or odd polynomial, depending on the parity of $j$. Thus, there are polynomials $A^d_k,B^d_k \in\RR[x]$ of degree $k$ such that $F^d_{2k}(x)=A^d_k(x^2)$ and $F^d_{2k+1}(x)=B^d_k(x^2)\cdot x$.

\begin{lema}\label{lem:interlace}$\,$
\begin{enumerate}[(i)]
             \item The polynomial $A^d_k$ has only simple, real and nonpositive roots for all $0\leq2k\leq d-1$. Moreover, for $0\leq2k\leq d-3$ the polynomials $A^d_k$ and $A^d_{k+1}$ are coprime and interlace.
             \item The polynomial $B^d_k$ has only simple, real and nonpositive roots for all $0\leq2k+1\leq d-1$. Moreover, for $0\leq2k+1\leq d-3$ the polynomials $B^d_k$ and $B^d_{k+1}$ are coprime and interlace.
            \end{enumerate}
\end{lema}

\begin{proof}
 We will prove $(i)$. The proof of $(ii)$ is verbatim the same.
 Since the real part of all the roots of $\tilde{H}^d_{2k}$ is $-\frac{d-j}{2}$, the zeros of $F^d_{2k}(x)$ are located on the imaginary axis. This implies that the roots of $A^d_k$ must be real and nonpositive. We show the rest of the claim by induction on $k$. For $k=0$ the statement is obviously true. Assume that $A^d_{k-1}$ and $A^d_{k}$ are coprime and interlace. Then it is immediate from the identity \[
  A_{k+1}^d=(4x+4 d k+d-8 k^2-6 k-2)\cdot A_k^d+2k(2k-1)(4x-(d-2k)^2)\cdot A_{k-1}^d
 \]that also $A^d_{k}$ and $A^d_{k+1}$ are coprime. But the identity also implies that $A^d_{k}$ and $A^d_{k+1}$ interlace by \cite[Lem. 1.82]{fisk06}.
\end{proof}

Let $d$ be a positive integer and $c\in\RR$. In the following we consider the polynomial
 \[
  G^d_c=H_{d-3}^d(x)+c\cdot H_{d-3}^d(x-1)+H_{d-3}^d(x-2)=H_{d-1}^d(x)+(c-2)\cdot H_{d-3}^d(x-1).
 \]

\begin{lema}
 Let $d$ be a positive odd integer. Then $G^d_c(-\frac{1}{2})=0$ if and only if $c=4d-6$.
\end{lema}

\begin{proof}
 Using (\ref{eqn:hypgeo}) one checks that \[G^d_c(-\frac{1}{2})=H_{d-3}^d(-\frac{1}{2})+c\cdot H_{d-3}^d(-\frac{3}{2})+H_{d-3}^d(-\frac{5}{2})=\frac{\Gamma(\frac{d}{2}-1)}{8 \sqrt{\pi} (\frac{d-1}{2})!}\cdot((4d-6)-c).\qedhere\]
\end{proof}

\begin{lema}
Let $d$ be a positive even integer and let $a_n$ be the coefficient of the linear term in $F^d_n$. For
$1\leq n\leq d-3$ odd we have:
$$\frac{a_{n+2}}{a_{n}}>-(n+1)^2+(d-2)(n+1)+d.$$
In particular, $a_{d-1}/a_{d-3}>d$.
\end{lema}
\begin{proof}
The proof is by induction on $n$, with $n=1$ following from $a_3/a_1=3d-6>3d-8$.
Let $f_n:=a_{n+2}/a_{n}$.
By Lemma \ref{lem:recpol} we have:
$$f_{n}=-2+d-3n+2dn-2n^2-n(n-1)(d-n)^2\cdot f_{n-2}^{-1}.$$
By induction we know that:
$$-f_{n-2}^{-1}\geq ((n-1)^2-(d-2)(n-1)-d)^{-1}=-(n(d-n)+1)^{-1}.$$
Hence,
\begin{align*}
f_{n}&\geq -2+d-3n+2dn-2n^2-(n-1)(d-n)\frac{n(d-n)}{n(d-n)+1}\\
&\geq -2+d-3n+2dn-2n^2-(n-1)(d-n)\\
&=-2-4n+2d+dn-n^2
=(-(n+1)^2+(d-2)(n+1)+d)+1 .\qedhere
\end{align*}
\end{proof}

\begin{lema}
 Let $d$ be a positive even integer. If $G^d_c$ has a double zero at $-\frac{1}{2}$, then $c>4d+2$.
\end{lema}

\begin{proof}
 First note that
 $$G_c^d=H_{d-1}^d(x)+(c-2)\cdot H_{d-3}^d(x-1)=F_{d-1}^d(x+\frac{1}{2})+(c-2)\cdot x (x+1)\cdot F_{d-3}^d(x+\frac{1}{2}).$$
 This has a double zero at $-\frac{1}{2}$ if and only if the linear term of
 $$F_{d-1}^d(x)+(c-2)\cdot (x-\frac{1}{2}) (x+\frac{1}{2})\cdot F_{d-3}^d(x)$$ vanishes. In the notation of the preceding lemma this implies that \[\frac{c-2}{4}=\frac{a_{d-1}}{a_{d-3}}>d.\qedhere\]
\end{proof}

\begin{lema}
 The polynomial $G^d_c$ has degree $d-1$ if and only if $c\neq -2$.
\end{lema}

\begin{proof}
 The leading coefficient of $H_{j}^d$ is $\frac{2^{j}}{(d-1)!}$. Thus, the degree drops if and only if \[2^{d-1}+(c-2)\cdot 2^{d-3}=0.\qedhere\]
\end{proof}

\begin{thm}\label{thm:whenGroots} Let $d\geq 3$.
 For every $-2\leq c\leq4d-6$ if $d$ is odd and for every $-2\leq c\leq4d+2$ if $d$ is even, the polynomial $G_c^d$ has only roots with real part equal to $-\frac{1}{2}$.
\end{thm}

\begin{proof}
 The claim is true for $G_2^d=H^d_{d-1}$ by \cite{rod02}. Since $(-1)^{d-1}G^d_c(x)=G^d_c(1-x)$ the zeros of $G^d_c$ are located symmetrically with respect to $R=R_{\frac{1}{2}}$. Since by Lemma \ref{lem:interlace} the zeros of $H^d_{d-1}$ are simple there are real numbers $a<b$ such that $2\in [a,b]$, $G_c^d$ has all its zeros on $R$ whenever $c\in[a,b]$ and $G_{a}^d$ and $G_{b}^d$ have either a multiple zero or degree less than $d-2$. We will show that such a multiple zero must be at $-\frac{1}{2}$. This will imply the claim by the preceding lemmas.
 
 If $d$ is even, then both $H^d_{d-1}(x)$ and $H^d_{d-3}(x-1)$ have a zero at $-\frac{1}{2}$. In that case let \[f(x)=\frac{H^d_{d-1}(x)}{2x+1}\textrm{ and } g(x)=\frac{H^d_{d-3}(x-1)}{2x+1}.\] If $d$ is odd, let $f(x)=H^d_{d-1}(x)$ and $g(x)=H^d_{d-3}(x-1)$. Let $\alpha_1,\ldots,\alpha_{2k}\in R$ with $\Imi(\alpha_{i})<\Imi(\alpha_{i+1})$ for all $1\leq i\leq 2k-1$ be the zeros of $f$. Two of the zeros of $g$ are $0$ and $-1$. Let $\beta_1,\ldots,\beta_{2k-2}\in R$ with $\Imi(\beta_{i})<\Imi(\beta_{i+1})$ for all $1\leq i\leq 2k-3$ be the remaining zeros of $g$. Since $f$ and $g$ are real polynomials we have that $\Imi(\alpha_{2k-i+1})=-\Imi(\alpha_i)$ for all $1\leq i\leq k$ and $\Imi(\beta_{2k-i-1})=-\Imi(\beta_i)$ for all $1\leq i\leq k-1$. By Lemma \ref{lem:interlace} we have furthermore \[\Imi(\alpha_1)<\Imi(\beta_1)<\Imi(\alpha_2)<\cdots<\Imi(\beta_{k-1})<\Imi(\alpha_k)<0,\]\[0<\Imi(\alpha_{k+1})<\Imi(\beta_{k})<\Imi(\alpha_{k+2})<\cdots<\Imi(\beta_{2k-2})<\Imi(\alpha_{2k}).\] Since $f$ and $g$ are coprime, $f+(c-2)g$ and $g$ are coprime for all $c\in\RR$. Thus, for all $c\in\RR$ and $1\leq i\leq k-2$ there are zeros $\gamma_i, \gamma_{k+i-1} \in L$ of $f+(c-2)g$ with $\Imi(\beta_i)<\Imi(\gamma_i)<\Imi(\beta_{i+1})$ and $\Imi(\beta_{k+i-1})<\Imi(\gamma_{k+i-1})<\Imi(\beta_{k+i})$. Furthermore, for $c>-2$ there is one zero of $f+(c-2)g$ with imaginary part larger than $\Imi(\beta_{2k-2})$ and with imaginary part smaller than $\Imi(\beta_1)$. Thus, the only possibility for $G_{a}^d$ or $G_{b}^d$ to have a multilple root is at $-\frac{1}{2}$.
\end{proof}


\begin{proof}[Proof of Theorem \ref{thm:any(2,n)}]
First note that $H_{2,n}=G^{n+2}_{2n}$ and $HS_{2,n}(t)=\frac{(1+t)^{n-1}(1+2nt+t^2)}{(1-t)^{n+2}}$ -cf.~Proposition \ref{prop:HS}. Any bipartite graph $G$ of type $(2,n)$ can be obtained from a complete $(2,m)$-bipartite graph for $m\leq n$ by adding vertices of degree one. Such an extension of graphs corresponds to multiplying the Hilbert series by $\frac{1+t}{1-t}$. Thus, $HS_G(t)=\frac{(1+t)^{n-1}(1+2mt+t^2)}{(1-t)^{n+2}}$ and the conclusion follows by Theorem \ref{thm:whenGroots}.
\end{proof}

\begin{exm}
 In general not every Ehrhart polynomial coming from a bipartite graph has its roots on $R$, e.g. let $G$ be the eight-cycle. The corresponding Ehrhart polynomial is
 \[
  H_G(x)=1+\frac{7}{2}x+\frac{175}{36}x^2+\frac{161}{36}x^3+\frac{35}{18}x^4+\frac{35}{36}x^5+\frac{7}{36}x^6+\frac{1}{18}x^7.
 \]
One checks that there is a root of $H_G$ having real part smaller than $-1$.
\end{exm}

\section{Dual polytopes - Examples}
In this section we present various results showing what happens for dual polytopes. First let us notice that it may happen that a polytope $P$ is reflexive, $H_P$ has roots on $R$ and $H_{P^*}$ does not have this property.
\begin{exm}
Consider $P_d$ to be the convex hull of $e_1,\dots,e_d,-e_1-\dots-e_d$. We have: $$HS_{P_d}(t)=\frac{\sum_{i=1}^d t^i}{(1-t)^{d+1}}.$$
In particular, all roots of the numerator belong to the unit circle and hence by \cite{rod02} all roots of $H_{P_d}$ belong to $R$. On the other hand:
$$H_{P_d^*}(x)={{(d+1)x+d}\choose d},$$
as $P_d^*$, up to a lattice shift, is the $(d+1)$st dilation of the standard $d$-dimensional simplex. 
\end{exm}
\begin{exm}
The dual $Cr_d^*$ of the Cross Polytope in Example \ref{cross} is simply the cube $[-1,1]^d$. In particular, $H_{Cr_d^*}(x)=(2x+1)^d$ with all roots equal to $-\frac{1}{2}$.
\end{exm}
We finish with an example dual to \ref{typeA}. In this case, the Ehrhart polynomials do not form an OPS. Yet, as we will see the interlacing property holds.
\begin{lema}\label{lem:dualA}
Let ${\bf A}_d^*$ be the $d$ dimensional dual polytope to the convex hull of the root system of type $A$. Then:
$$H_{{\bf A}_d^*}(m)=\sum_{i=0}^d {{d+1}\choose i}m^i
.$$
\end{lema}
\begin{proof}
It is enough to prove that $f(m):=\sum_{k=1}^m H_{{\bf A}_d^*}(k)=(m+1)^{d+1}$. 
Consider a transformation:
$$t:\Z^{d}\ni(a_1,\dots,a_d)\rightarrow (0,a_1,a_1+a_2,\dots,a_1+a_2+\dots+a_d)\in \{0\}\times \Z^{d}.$$
An integral point $a$ belongs to $k{\bf A}_d^*$ if and only if any two coordinates of $t(a)$ differ (in absolute value) by at most $k$, or put differently, the coordinates of $t(a)$ belong to an interval $[a,a+k]$ for some $a\in \Z$. Notice that coordinates of $t(a)$ belongs to an interval of length $b$ if and only if they belong to $k-b+1$ intervals of type $[a,a+k]$ for different $a\in \Z$. Thus, summing as multisets all integral points in $t(k{\bf A}_d^*)$ for $k=0,\dots,m$ we enumerate integral sequences $(0,c_1,c_2,\dots,c_d)$, each one counted that many times as many intervals $[a,a+m]$ contain all $c_i$'s. Hence, by double counting and looking at all possible intervals we obtain:
$$f(m)=\sum_{a=-m}^0 (m+1)^d=(m+1)^{d+1}.$$ 
\end{proof}
\begin{cor}
The Ehrhart polynomial of the dual polytope ${\bf A}_d^*$ to the convex hull of the root system of type $A$ has all roots on $R$. The roots of $H_{{\bf A}_d^*}$ and $H_{{\bf A}_{d+1}^*}$ interlace on $R$.
\end{cor}
\begin{proof}
By Lemma \ref{lem:dualA} we have $H_{{\bf A}_d^*}(m)=(m+1)^{d+1}-m^{d+1}$. Thus the roots of $H_{{\bf A}_d^*}$ are the inverses of the (nonzero) $(d+1)$st roots of unity shifted by $-1$. The line $R$ is the inverse of the circle of radius one centered at $-1$. The corollary follows, as $d$th and $(d+1)$st roots of unity 'interlace' on the unit circle. 
\end{proof}
We believe that further nice results will be obtained in future for the associahedra and duals of graph polytopes. 
\section{Appendix}\label{Appendix}
We show how to efficiently determine recursive relations on the (most complicated) example (\ref{recH3n}).
\begin{verbatim}
In: Num:=Numerator[Simplify[(a*t*D[H2[k+1,t],t]
+b*H2[k +1,t]+c*H2[k,t]+d*H1[k+1,t])/(H3[k+1,t])]]
In: Num/.t->0
Out: b+c+d
In: (Simplify[(Num-(b+c+d))/t])/.t->0
Out: -2c-b(-2-2k)+2ck+a(5+4k)
In: (Simplify[(((Num-(b+c+d))/t)-(-2c-b(-2-2k)+2ck+a(5+4k)))/t])/.t->0
Out: -2d+c(2-4k)+a(11+10k+4k^2)
In: (Simplify[(((((Num-(b+c+d))/t)-(-2c-b(-2-2k)+2ck+a(5+4k)))/t)-
(-2d+c(2-4k)+a(11+10k+4k^2)))/t])/.t->0
Out: -2c+2ck-b(2+2k)+a(7+6k)
In: Denominator[Simplify[(a*t*D[H2[k+1,t],t]+b*H2[k+1,t]
+c*H2[k,t]+d*H1[k+1,t])/(H3[k+1,t])]]
Out: 1+4kt+4kt^3+t^4+3t^2-kt^2+3k^2t^2
In: Solve[{b+c+d==1,-2c-b(-2-2k)+2ck+a(5+4k)==4k+4,-2d+c(2-4k)+a(11+10k+4k^2)
==6+5k+3*k^2,-2c+2ck-b(2+2k)+a(7+6k)==4k+4,a+d+c-b==1},{a,b,c,d}]
Out:{{a->-((-16-13k-3k^2)/(4(6+5k+k^2))),b->-((-16-13k-3k^2)/(8(6+5k+k^2))),
c->-((-18k-13k^2-k^3)/(8(-1+k)(6+5k+k^2))),
d->-((32+13k-9k^2-4k^3)/(8(-1+k)(6+5k+k^2)))}}
\end{verbatim}
\bibliographystyle{amsalpha}
\bibliography{Xbib}
\end{document}